\documentclass[11pt,leqno]{article}
\usepackage{verbatim}
\usepackage[english]{babel}
\usepackage{enumitem}
\usepackage{tikz}
\usepackage{graphicx}
\usepackage{booktabs,tabularx}
\usepackage{verbatim}
\usepackage{etoolbox}
\usepackage{tkz-euclide}
\usepackage{geometry}
\usepackage{mathtools}
\usepackage{color}
\usepackage[title]{appendix}
\usepackage[all]{xy}
\usepackage{tikz-cd}
\usepackage{blindtext}
\usepackage[T1]{fontenc}
\usepackage{amsthm}
\usepackage{amsfonts}
\usepackage{txfonts}
\usepackage{palatino,amssymb,epsfig}
\usepackage{latexsym,epsf,epic,amscd}
\usepackage{mathrsfs}
\usepackage{graphicx}
\usepackage{caption}
\usepackage{subcaption}

\usepackage{float}

\usepackage{doi} 
\usepackage{hyperref}

\usepackage{appendix}
\usepackage{amsmath}
\usepackage{bookmark} 
\usepackage[nottoc]{tocbibind}
\hypersetup{
	colorlinks=true,
	linkcolor=black,
	citecolor=black,
	anchorcolor=black,
	urlcolor=black}
\allowdisplaybreaks[4]
\numberwithin{equation}{section}
\DeclareFontFamily{OMX}{yhex}{}
\DeclareFontShape{OMX}{yhex}{m}{n}{<->yhcmex10}{}
\DeclareSymbolFont{yhlargesymbols}{OMX}{yhex}{m}{n}
\DeclareMathAccent{\wideparen}{\mathord}{yhlargesymbols}{"F3}

\newcommand{\Addresses}{{
		\bigskip
		\footnotesize
	
 \par\nopagebreak
    \medskip
  \noindent Yu Sun, \href{yusun15185105160@163.com}{yusun15185105160@163.com}
  \newline\textit{School of Mathematics and Physics, Nanjing Institute of Technology, Nanjing, 211100, P.R. China.} 
		
      \noindent Lei Bao, \href {1025081901@njupt.edu.cn }{1025081901@njupt.edu.cn }
		\newline\textit{ School of Science, Nanjing University of Posts and Telecommunications,  Nanjing, 210003, P.R. China.}
	
        \noindent Wenhao Fu, \href{b22090210@njupt.edu.cn}{b22090210@njupt.edu.cn}
		\newline\textit{ School of Science, Nanjing University of Posts and Telecommunications,  Nanjing, 210003, P.R. China.}

     \noindent Zunwu He$^{*}$ (Corresponding author), \href{Hzwmath789@scut.educ.cn}{hzwmath789@scut.edu.cn }
		\newline\textit{  School of Mathematics, South China University of Technology, Guangzhou, 510641, P.R. China.}  }}

\date{}
\newtheorem{theorem}{Theorem}[section]
\newtheorem{lemma}[theorem]{Lemma}
\newtheorem{proposition}[theorem]{Proposition}
\newtheorem{corollary}[theorem]{Corollary}
\theoremstyle{definition}
\newtheorem{definition}[theorem]{Definition}
\newtheorem{remark}[theorem]{Remark}

\title{Existence, Rigidity, and Discrete Schwarz--Pick Lemma for Generalized Hyperbolic Circle Packings with Boundary and Discrete Gaussian Curvatures}
\author{Yu Sun, Lei Bao, Wenhao Fu, Zunwu He$^{*}$}
\date{}

\begin{document}
	
	\maketitle
	\begin{abstract}

This paper is concerned with generalized hyperbolic circle packings on compact bordered surfaces, endowed with finite polygonal cellular decompositions.
We investigate the problem of realizing generalized hyperbolic circle packings with prescribed geodesic curvatures at boundary vertices, prescribed total geodesic curvatures at interior vertices, and prescribed discrete Gaussian curvatures at the centers of dual circles. We give a necessary and sufficient condition for the existence of such generalized hyperbolic circle packings and show their uniqueness. We also establish a discrete Schwarz--Pick lemma in this setting, including comparison results for vertex curvatures, generalized circle arc lengths, distances, and areas, together with the corresponding rigidity statements.

	\end{abstract}

	\section{Introduction}
	\label{sec:intro}

\subsection{Background}
\label{subsec:background}

Discrete conformal structures on polyhedral surfaces are often regarded as discrete analogues of conformal structures on smooth surfaces, and circle packings play a central role in this setting. Thurston \cite{thurstongeometry} used circle packings on triangulated closed surfaces to study hyperbolic structures on \(3\)-manifolds, revealing a deep relation between circle packings and conformal geometry. For a general introduction to this subject, see also Stephenson \cite{stephenson2005introduction}.

A basic tool in the study of circle packings and circle patterns is the variational principle. In 1991, Colin de Verdi\`ere \cite{de1991principe} introduced a variational approach to circle packings, which has since become a standard method for proving existence and rigidity. Important developments in this direction were made by Bobenko and Springborn \cite{bobenko2004variational}, Leibon \cite{leibon2002characterizing}, Rivin \cite{rivin1994euclidean}, and Luo \cite{Luo12,Luo13,Luo}.

In spherical background geometry, however, the known functionals were non-convex, which made the variational method difficult to apply. This difficulty was recently overcome by Nie \cite{nie2024circle}, who introduced a convex functional involving total geodesic curvatures and proved existence and rigidity for circle patterns with prescribed total geodesic curvatures at vertices. In hyperbolic background geometry, the first author and her collaborators \cite{ba2023circle} studied generalized circle packings with conical singularities on triangulated compact surfaces. Their setting includes circles, horocycles, and hypercycles, with prescribed total geodesic curvatures at vertices.

For surfaces with boundary, the authors in \cite{HLLY} proved existence and rigidity for generalized circle packings with prescribed geodesic curvatures on boundary vertices and prescribed total geodesic curvatures on interior vertices. More recently, the authors in \cite{hugeneralized} extended this result from triangulations to finite polygonal cellular decompositions by using the Perron method. Building on \cite{ba2023circle,hugeneralized, HLLY}, the present paper studies generalized circle packings on surfaces with finite polygonal cellular decompositions and establishes existence and rigidity for generalized circle packing metrics with prescribed geodesic curvatures on boundary vertices, prescribed total geodesic curvatures on interior vertices, and prescribed conical angles, equivalently, prescribed discrete Gaussian curvatures, at the centers of dual circles.

The Schwarz--Pick phenomenon also has a discrete counterpart in circle packing theory. In the classical planar case, Beardon and Stephenson \cite{beardon1991schwarz} established a discrete Schwarz--Pick lemma for circle packings. Later, this comparison principle was extended to generalized circle packings on triangulated surfaces. In the present paper, we further extend it to the setting considered above, namely, generalized circle packings on surfaces with boundary and finite polygonal cellular decompositions, with prescribed discrete Gaussian curvatures at the centers of dual circles.

	\subsection{Set up}
	\label{subsec:setup}

	\subsubsection{Finite polygonal cellular decomposition}
\label{subsubsec:finite_polygonal_cellular_decomposition}

Let \(S_{g,n}\) be a surface of genus \(g\) with \(n\) boundary components. More precisely, \(S_{g,n}\) is obtained by removing \(n\) pairwise disjoint open disks from an oriented compact surface of genus \(g \ge 0\) without boundary.

Consider an embedding \(\eta:(V_{g,n},E_{g,n})\to S_{g,n}\) of a graph \((V_{g,n},E_{g,n})\) into \(S_{g,n}\). We assume that:
\begin{itemize}
    \item each boundary component of \(S_{g,n}\) contains at least one vertex of \(V_{g,n}\);
    \item each boundary component of \(S_{g,n}\) is a union of edges in \(E_{g,n}\).
\end{itemize}

For simplicity, we identify \(V_{g,n}\) with \(\eta(V_{g,n})\) and \(E_{g,n}\) with \(\eta(E_{g,n})\). When no confusion arises, we omit the subscripts \(g\) and \(n\).

The faces of the embedded graph are the connected components of \(S_{g,n}\setminus (V_{g,n}\cup E_{g,n})\). We denote the set of all faces by \(F_{g,n}\).

The embedding is called a \emph{closed \(2\)-cell embedding} if:
\begin{itemize}
    \item[(I)] the closure of each face is homeomorphic to a closed disk;
    \item[(II)] the boundary of each face is a simple closed curve consisting of finitely many edges.
\end{itemize}

Let \(\Sigma_{g,n}\) be the cellular decomposition of \(S_{g,n}\) induced by such a closed \(2\)-cell embedding, whose \(1\)-skeleton is \((V_{g,n},E_{g,n})\). Then \(V_{g,n}\), \(E_{g,n}\), and \(F_{g,n}\) are the sets of \(0\)-cells, \(1\)-cells, and \(2\)-cells, respectively.

\begin{definition}
\label{def:finite_polygonal_cellular_decomposition}
(Finite Polygonal Cellular Decomposition)
A finite cellular decomposition is called a finite polygonal cellular decomposition if:
\begin{itemize}
    \item[(I)] every \(2\)-cell is a polygon;
    \item[(II)] every \(0\)-cell is incident to at least three \(1\)-cells;
    \item[(III)] the \(1\)-skeleton is a simple graph, that is, it has neither loops nor multiple edges.
\end{itemize}
\end{definition}

\begin{remark}
\label{rem:finite_polygonal_cellular_decomposition_symbols}
We use the following notation.

For a face \(P\in F_{g,n}\), let \(V(P)\) be the set of vertices of \(P\), and let \(N(P)\) be the number of vertices of \(P\).

For a subset \(W\subset V_{g,n}\), define
\[
F_W=\{P\in F_{g,n}\mid V(P)\cap W\neq \emptyset\}.
\]
For \(P\in F_W\), define
\[
N(P,W)=\#(V(P)\cap W).
\]

We write \(V^\circ\) for the set of interior vertices and \(V^\partial\) for the set of boundary vertices. Thus \(V=V^\circ\cup V^\partial\).
\end{remark}

\subsubsection{Generalized circle packings on surfaces}
\label{subsubsec:generalized_circles}

This subsection uses some basic facts from hyperbolic geometry. In the Poincar\'e disk model \(\mathbb H^2\), a generalized circle may be a circle, a horocycle, or a hypercycle. The center of a horocycle is its unique ideal point, while the center of a hypercycle is the geodesic joining its two ideal endpoints; see Figure~\ref{f1}.

\begin{figure}[h]
    \centering
    \includegraphics[width=0.4\textwidth]{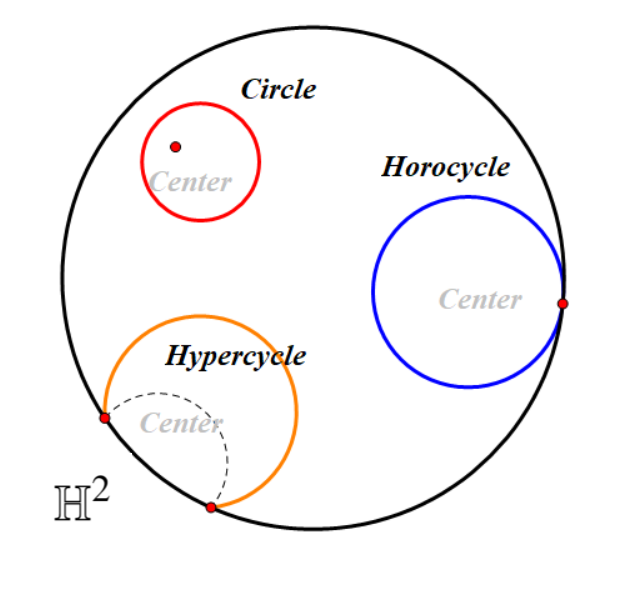}
    \caption{A circle, a hypercycle, and a horocycle in \(\mathbb H^2\).}
    \label{f1}
\end{figure}

For a generalized circle, the radius is the distance from a point on the circle to its center, and this distance is constant. A generalized circle is determined by its constant geodesic curvature \(k\). The relation between the radius \(r\) and \(k\) is
\[
r(k)=
\begin{cases}
\operatorname{arctanh} k, & 0<k<1,\\
+\infty, & k=1,\\
\operatorname{arccoth} k, & k>1.
\end{cases}
\]
Here \(k<1\) corresponds to a hypercycle, \(k=1\) to a horocycle, and \(k>1\) to a circle.

\begin{figure}[ht]
    \centering
    \includegraphics[width=0.3\textwidth]{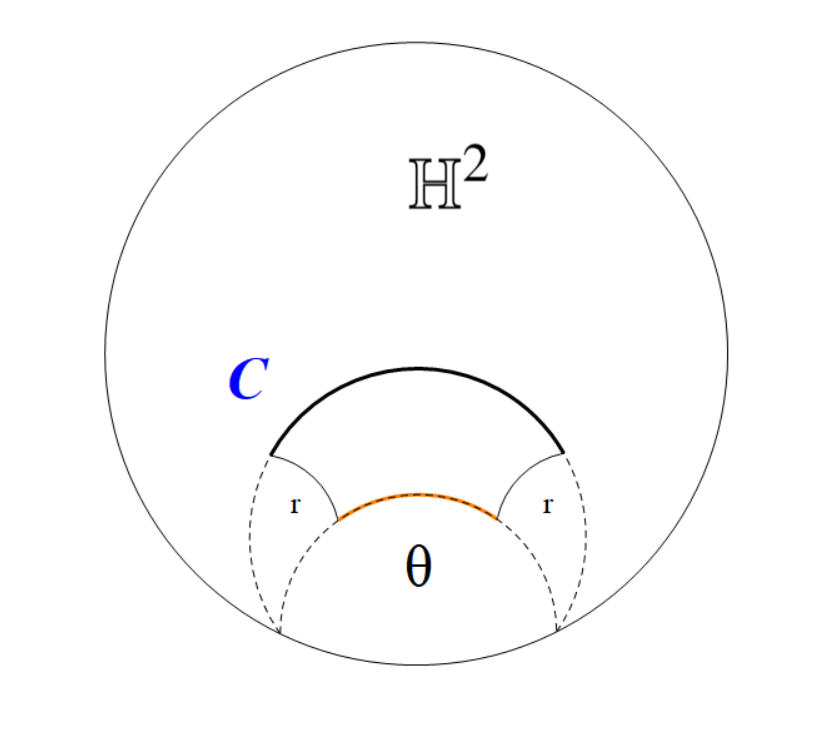}
    \captionof{figure}{The interior angle \(\theta\) of an arc \(C\) of a hypercycle.}
    \label{f2}
\end{figure}

For a circle or a horocycle, the interior angle \(\theta\) of an arc is the angle between the two shortest geodesics joining the endpoints of the arc to the center. For a hypercycle, \(\theta\) is defined as the hyperbolic length of the projection of the arc onto its geodesic axis \(\gamma\); see Figure~\ref{f2}. Table~\ref{t1} lists the relations among the radius \(r\), the geodesic curvature \(k\), the interior angle \(\theta\), and the arc length.

\begin{table}[H]
\centering
\begin{tabular}{|c|c|c|c|}
\toprule
 & \textbf{Radius} & \textbf{Geodesic curvature} & \textbf{Arc length} \\
\midrule
Circle & \(0<r<\infty\) & \(k=\coth r\) & \(l=\theta \sinh r\) \\
Horocycle & \(r=\infty\) & \(k=1\) & / \\
Hypercycle & \(0<r<\infty\) & \(k=\tanh r\) & \(l=\theta \cosh r\) \\
\bottomrule
\end{tabular}
\caption{Relations among the radius, the geodesic curvature, and the arc length.}
\label{t1}
\end{table}

According to Lemma 1.1 in \cite{hu2025combinatorial}, we have the following result for a single polygon.

\begin{lemma}
\label{lem:generalized_circle_packing_polygon}
(Generalized circle packing on a polygon)
Let \(P\) be an abstract polygon with vertices \(v_1,\dots,v_{N(P)}\) arranged counterclockwise. Given \((k_1,\dots,k_{N(P)})\in \mathbb R_{>0}^{N(P)}\) on \(V(P)\) and a number \(Y_P\) satisfying \((2-N(P))\pi<Y_P<2\pi\), there exists a geometric pattern \(\mathcal P\) formed by \(N(P)\) generalized circles \(C_1,\dots,C_{N(P)}\) in \(\mathbb H^2\) such that:
\begin{itemize}
    \item[(I)] the geodesic curvature of \(C_i\) is \(k_i\) for \(i=1,\dots,N(P)\);
    \item[(II)] the circles \(C_i\) and \(C_{i+1}\) are tangent, where \(C_{N(P)+1}:=C_1\);
    \item[(III)] there exists a hyperbolic circle \(C_P\) with conical angle \(\alpha_P=2\pi-Y_P\) at its center, such that \(C_P\) is perpendicular to each \(C_i\) and passes through all tangent points of \(C_1,\dots,C_{N(P)}\).
\end{itemize}
We call \(C_P\) the dual circle of \(P\). If \(k_P\) denotes the geodesic curvature of \(C_P\), then \(k_P\) is a \(C^1\)-function of \((k_1,\dots,k_{N(P)})\).
\end{lemma}

The pattern \(\mathcal P\) is called a generalized circle packing on \(P\) with geodesic curvatures \((k_1,\dots,k_{N(P)})\) and conical angle \(\alpha_P=2\pi-Y_P\). The number \(Y_P\) is called the discrete Gaussian curvature at the center of the dual circle \(C_P\).

\begin{definition}
\label{def:generalized_circle_packing}
(Generalized circle packing metric)
Let \(S_{g,n}\) be a surface with polygonal decomposition \(\Sigma_{g,n}=(V,E,F)\). For each face \(P\in F\), choose a discrete Gaussian curvature \(Y_P\in ((2-N(P))\pi,2\pi)\). Let \(k\in \mathbb R_{>0}^{|V|}\) be the geodesic curvatures on \(V\), where \(k_v\) is the geodesic curvature at \(v\in V\). We construct a new surface \(\widetilde S_{g,n}\) as follows:
\begin{itemize}
    \item for each face \(P\in F\) with vertices \(v_1,\dots,v_{N(P)}\), construct a generalized conical polygon \(\widehat P\) with conical angle \(\alpha_P=2\pi-Y_P\) and vertex geodesic curvatures \((k_{v_1},\dots,k_{v_{N(P)}})\);
    \item glue these generalized conical polygons along the corresponding edges.
\end{itemize}
The resulting surface \(\widetilde S_{g,n}\) carries a hyperbolic conical metric, called a generalized circle packing metric.
\end{definition}

\begin{definition}[Total geodesic curvature]
\label{def:total_geodesic_curvature}
Let \(C_i\) be the generalized circle centered at a vertex. Its total geodesic curvature is defined by
\[
L_i=\int_{C_i} k_i\,ds,
\]
where \(k_i\) is the geodesic curvature of \(C_i\), and \(ds\) is the arc-length element along \(C_i\).

If \(C_i\) is only an arc contained in a face \(P\), let \(l_{i,P}\) be the length of this arc. Then
\[
L_{i,P}=\int_{C_i\cap P} k_i\,ds=l_{i,P}k_i.
\]
\end{definition}

\subsection{Main results}
\label{subsec:main_results}

\subsubsection{Boundary value problem}
\label{subsubsec:bvp}

We study generalized circle packings on \(S_{g,n}\) with prescribed boundary geodesic curvatures, prescribed interior total geodesic curvatures, and prescribed discrete Gaussian curvatures on dual circles.

Our first main result gives the existence and uniqueness of such metrics.

\begin{theorem}
\label{thm:existence_uniqueness_gcp}
Let \(S_{g,n}\) be a surface of genus \(g\) with \(n\) boundary components, and let \(\Sigma_{g,n}=(V,E,F)\) be a finite polygonal cellular decomposition of \(S_{g,n}\). Then there exists a surface \(\widetilde S_{g,n}\) with a generalized circle packing metric having interior total geodesic curvatures \(L_1,\dots,L_{|V^\circ|}\), boundary geodesic curvatures \(k_1,\dots,k_{|V^\partial|}\), and discrete Gaussian curvatures \(Y_1,\dots,Y_{|F|}\) at the centers of the dual circles if and only if \(Y_P\in ((2-N(P))\pi,2\pi)\) for every \(P\in F\) and \((L_1,\dots,L_{|V^\circ|})^T\in \mathcal L\), where
\[
\mathcal L=
\left\{
(L_1,\dots,L_{|V^\circ|})^T\in \mathbb R_{>0}^{|V^\circ|}
\;\middle|\;
\sum_{v\in W}L_v
<
\sum_{P\in F_W}
\pi \min \left\{N(P,W),\,N(P)-2+\frac{Y_P}{\pi}\right\},
\ \forall\, W\subset V^\circ
\right\}.
\]
Moreover, the generalized circle packing metric is unique if it exists.
\end{theorem}

\subsubsection{Discrete Schwarz--Pick lemma}
\label{subsubsec:schwarz_pick}

The classical Schwarz--Pick lemma says that a holomorphic self-map of the unit disk does not increase hyperbolic distance. In the discrete setting, Beardon and Stephenson proved an analogue for hyperbolic circle packings. Later, this was extended to generalized circle packings on triangulated surfaces. Here we consider the same question for surfaces with finite polygonal cellular decompositions.

Let \(\mathcal P^1\) and \(\mathcal P^2\) be two generalized circle packings on the same surface. Assume that they have the same prescribed interior total geodesic curvatures and the same discrete Gaussian curvatures on dual circles. If the boundary geodesic curvatures of \(\mathcal P^1\) are no greater than those of \(\mathcal P^2\), then one can compare the vertex curvatures, distances, areas, and arc lengths of the two packings.

These comparisons are stated in the next theorem and will be proved in Section~\ref{sec:schwarz_pick_lemma}.

\begin{theorem}[Discrete Schwarz--Pick lemma]
\label{thm:discrete_schwarz_pick}
For \(j=1,2\), let \(\hat k_i^j\) be the prescribed geodesic curvature on the boundary for \(v_i\in V^\partial\), and let \(\hat L_i\) be the prescribed total geodesic curvature at interior vertices \(v_i\in V^\circ\), where \((\hat L_1,\dots,\hat L_{|V^\circ|})^T\in \mathcal L\). Let \(\mathcal P^j\) be the generalized circle packing on \((S_{g,n},\Sigma_{g,n})\) realizing these data. Denote by \(k_i^j\) the geodesic curvature of \(\mathcal P^j\) at \(v_i\in V^\circ\), and set \(k_i^j=\hat k_i^j\) for \(v_i\in V^\partial\).

Assume that \(\hat k_i^1\le \hat k_i^2\) for all \(v_i\in V^\partial\). Then:
\begin{itemize}
    \item[(I)] \(k_i^1\le k_i^2\) for all \(v_i\in V\);
    \item[(II)] \(l_{i,P}^1\ge l_{i,P}^2\) for every face \(P\in F\) and every \(v_i\in V(P)\), where \(l_{i,P}^j\) is the length of the generalized circle arc centered at \(v_i\) in \(P\) for \(j=1,2\);
    \item[(III)] \(\rho^1(v_i,v_l)\ge \rho^2(v_i,v_l)\) for any \(v_i,v_l\in \widetilde V\), where \(\widetilde V=\{v_i\in V\mid k_i^1>1\}\), and \(\rho^1,\rho^2\) denote the distances defined by admissible arc chains associated with the polygonal cellular decomposition on the hyperbolic surfaces \(\widetilde S_{g,n}^1\) and \(\widetilde S_{g,n}^2\), respectively;
    \item[(IV)] \(\mathrm{Area}^1(\Omega_P)\ge \mathrm{Area}^2(\Omega_P)\) for every face \(P\in F\), where \(\mathrm{Area}^j(\Omega_P)\) is the hyperbolic area of the ideal region corresponding to \(P\) in \(\mathcal P^j\).
\end{itemize}
Moreover, if equality occurs in any one of the following cases:
\begin{itemize}
    \item[(I)] \(k_i^1=k_i^2\) at some interior vertex \(v_i\in V^\circ\);
    \item[(II)] \(l_{i,P}^1=l_{i,P}^2\) for some generalized circle arc centered at an interior vertex \(v_i\in V^\circ\);
    \item[(III)] \(\rho^1(v_i,v_l)=\rho^2(v_i,v_l)\) for some \(v_i,v_l\in \widetilde V\), with at least one of them in \(V^\circ\);
    \item[(IV)] \(\mathrm{Area}^1(\Omega_P)=\mathrm{Area}^2(\Omega_P)\) for some face \(P\in F\) containing an interior vertex,
\end{itemize}
then \(k_i^1=k_i^2\) for all \(v_i\in V\). In particular, \(\mathcal P^1=\mathcal P^2\).
\end{theorem}

    \section{The potential function}
\label{sec:potential}

In this section and the next one, we prove Theorem~\ref{thm:existence_uniqueness_gcp}. By Lemma~\ref{lem:generalized_circle_packing_polygon} and Definition~\ref{def:generalized_circle_packing}, the desired generalized circle packing is determined once the curvature \(k\) is known. Thus the problem is reduced to recovering \(k\) from the prescribed data \(\hat L\) and \(\hat k\). To do this, we use a variational principle similar to that in \cite{hu2025hyperbolic}.

\subsection{Potential functions for local generalized hyperbolic circle packings}
\label{subsec:potential_functions_local_gcp}

Let \(P\) be a face, and let \(\widetilde P\) be the associated generalized conical polygon with conical angle \(\alpha_P=2\pi-Y_P\). Suppose that the geodesic curvatures at the vertices of \(P\) are \(k_1,\dots,k_{N(P)}\). By Lemma 2.1 in \cite{hu2025combinatorial}, if all \(k_1,\dots,k_{N(P)}\) are variable, then the \(1\)-form
\[
\omega_P=\sum_{i=1}^{N(P)} l_{i,P}\,dk_i
\]
is closed.

Set \(L_{i,P}=l_{i,P}k_i\) and \(s_i=\ln k_i\). Then the \(1\)-form
\[
\widetilde\omega_P=\sum_{i=1}^{N(P)} L_{i,P}\,ds_i
\]
is also closed. Hence, if all \(k_1,\dots,k_{N(P)}\) are variable, we define the local potential function by
\[
\mathcal E_P(s_1,\dots,s_{N(P)})=\int_0^{(s_1,\dots,s_{N(P)})}\widetilde\omega_P.
\]

Now let
\[
I_P^0=\{\,i\in \{1,\dots,N(P)\}\mid k_i \text{ is not fixed}\,\}.
\]
Write \(I_P^0=\{i_1,\dots,i_m\}\). For simplicity, we relabel the corresponding variables by
\[
s_a:=s_{i_a},
\qquad
L_{a,P}:=L_{i_a,P},
\qquad
a=1,\dots,m.
\]
If some of the \(k_i\) are fixed, then we define
\[
\mathcal E_P(s_1,\dots,s_m)=\int_0^{(s_1,\dots,s_m)}\widetilde\omega_P,
\qquad
\widetilde\omega_P=\sum_{a=1}^m L_{a,P}\,ds_a.
\]
Since \(dk_i=0\) for every fixed \(k_i\), this definition is well defined.

We now prove that these local potential functions are strictly convex.

\begin{proposition}
\label{prop:potential_convex_local}
The local potential functions \(\mathcal E_P\) are strictly convex.
\end{proposition}

\begin{proof}
We first consider the case where all \(k_1,\dots,k_{N(P)}\) are variable. Then
\[
\nabla \mathcal E_P=(L_{1,P},\dots,L_{N(P),P})^T,
\qquad
\mathrm{Hess}\,\mathcal E_P=
\left(\frac{\partial L_{i,P}}{\partial s_j}\right)_{N(P)\times N(P)}.
\]
By Lemma 2.4 in \cite{hu2025combinatorial}, the matrix \(\mathrm{Hess}\,\mathcal E_P\) is positive definite. Hence \(\mathcal E_P\) is strictly convex.

Now assume that some of the \(k_i\) are fixed. Then
\[
\nabla \mathcal E_P=(L_{1,P},\dots,L_{m,P})^T,
\qquad
\mathrm{Hess}\,\mathcal E_P=
\left(\frac{\partial L_{a,P}}{\partial s_b}\right)_{m\times m}.
\]

Let \(\Omega_P\) be the region enclosed by the sub-arcs \(C_{i,P}\), \(i=1,\dots,N(P)\), as shown in Figure~\ref{f4}.

\begin{figure}[htbp]
 \centering
 \includegraphics[width=0.7\textwidth]{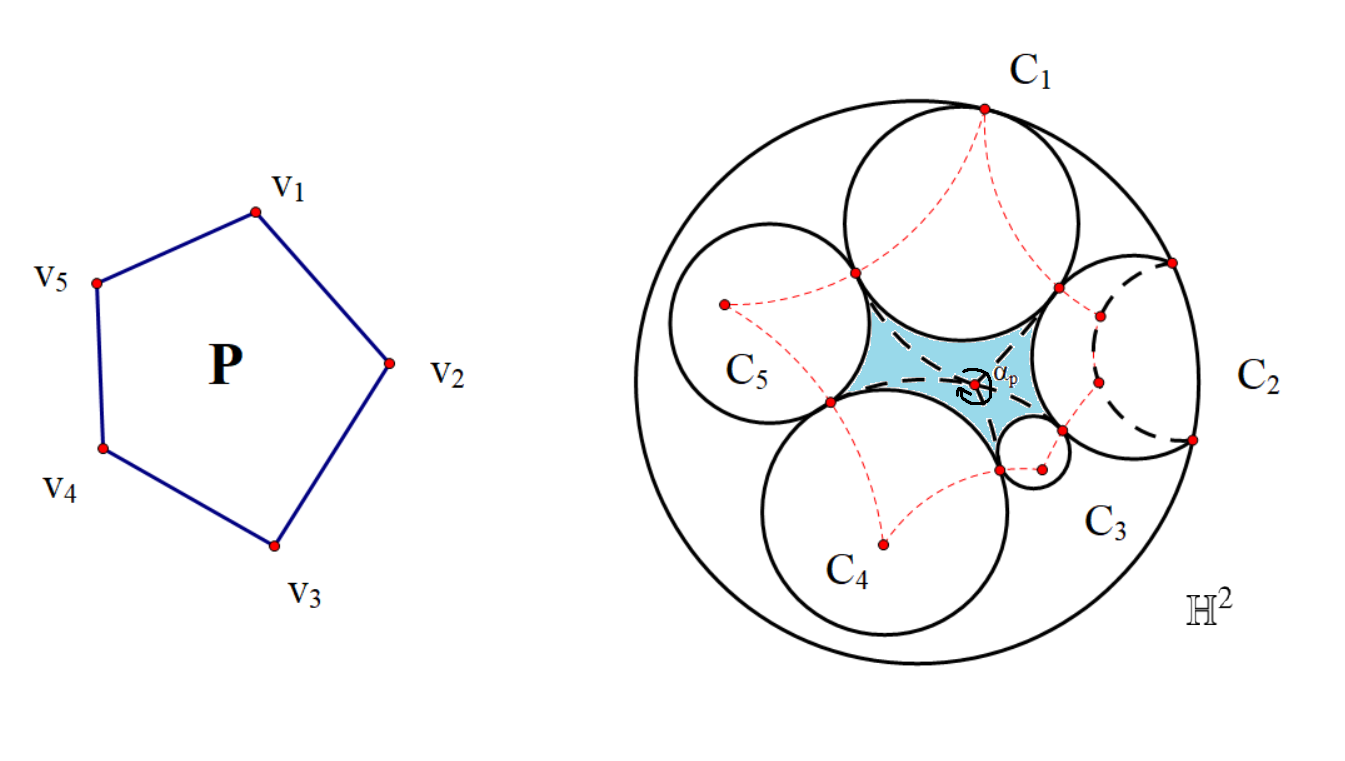}
 \caption{An example of face \(P\) and region \(\Omega_P\).}
 \label{f4}
\end{figure}

By the Gauss--Bonnet theorem,
\[
\mathrm{Area}(\Omega_P)=N(P)\pi-\alpha_P-\sum_{i=1}^{N(P)}L_{i,P}.
\]

By Lemma 2.3 and the proof of Lemma 2.4 in \cite{hu2025combinatorial}, we have
\[
\frac{\partial L_{a,P}}{\partial s_a}>0,
\qquad
\frac{\partial L_{a,P}}{\partial s_b}<0
\quad \text{for } a\neq b.
\]
Moreover, for each \(a=1,\dots,m\),
\[
\sum_{b=1}^m \frac{\partial L_{b,P}}{\partial s_a}
=
-\sum_{i\notin I_P^0}\frac{\partial L_{i,P}}{\partial s_a}
-\frac{\partial \mathrm{Area}(\Omega_P)}{\partial s_a}.
\]
Again by Lemma 2.3 and the proof of Lemma 2.4 in \cite{hu2025combinatorial},
\[
\frac{\partial L_{i,P}}{\partial s_a}=0
\quad \text{for } i\notin I_P^0,
\qquad
\frac{\partial \mathrm{Area}(\Omega_P)}{\partial s_a}<0.
\]
Therefore,
\[
\sum_{b=1}^m \frac{\partial L_{b,P}}{\partial s_a}>0
\qquad \text{for every } a=1,\dots,m.
\]

Hence \(\mathrm{Hess}\,\mathcal E_P\) is symmetric, strictly diagonally dominant, and has positive diagonal entries. It follows that \(\mathrm{Hess}\,\mathcal E_P\) is positive definite. Therefore \(\mathcal E_P\) is strictly convex in this case as well.
\end{proof}

\subsection{The potential function for generalized circle packings}
\label{subsec:potential_functions}

Let \(S_{g,n}\) be a surface with boundary, and let \(\Sigma_{g,n}=(V,E,F)\) be a finite polygonal cellular decomposition as in Definition~\ref{def:finite_polygonal_cellular_decomposition}.

We label the interior vertices in \(V^\circ\) by \(1,\dots,|V^\circ|\), and the boundary vertices in \(V^\partial\) by \(1,\dots,|V^\partial|\).

Fix prescribed boundary geodesic curvatures \((\hat k_1,\dots,\hat k_{|V^\partial|})\in \mathbb R_{>0}^{|V^\partial|}\). For each interior vertex \(v_i\in V^\circ\), define its total geodesic curvature by
\[
L_i=\sum_{P\in F_{\{v_i\}}}L_{i,P}.
\]

Set \(s_i=\ln k_i\) for \(i=1,\dots,|V^\circ|\), and \(\hat s_i=\ln \hat k_i\) for \(i=1,\dots,|V^\partial|\). Write \(s=(s_1,\dots,s_{|V^\circ|})\in \mathbb R^{|V^\circ|}\). For each face \(P\in F\), the boundary variables are fixed by \(\hat s\), so the corresponding local potential \(\mathcal E_P\) depends only on the interior variables of \(P\). We define the global potential function by
\[
\mathcal E(s)=\sum_{P\in F}\mathcal E_P.
\]

For \(1\le i\le |V^\circ|\), differentiating \(\mathcal E\) with respect to \(s_i\) gives
\[
\frac{\partial \mathcal E}{\partial s_i}
=
\sum_{P\in F_{\{v_i\}}}\frac{\partial \mathcal E_P}{\partial s_i}
=
\sum_{P\in F_{\{v_i\}}}L_{i,P}
=
L_i.
\]
Hence
\[
\nabla \mathcal E=(L_1,\dots,L_{|V^\circ|})^T.
\]
Its Hessian is the Jacobi matrix
\[
\mathrm{Hess}\,\mathcal E=M=
\begin{pmatrix}
\frac{\partial L_1}{\partial s_1} & \cdots & \frac{\partial L_1}{\partial s_{|V^\circ|}} \\
\vdots & \ddots & \vdots \\
\frac{\partial L_{|V^\circ|}}{\partial s_1} & \cdots & \frac{\partial L_{|V^\circ|}}{\partial s_{|V^\circ|}}
\end{pmatrix}.
\]

\begin{proposition}
\label{prop:potential_convex_g}
The Jacobi matrix \(M\) is positive definite.
\end{proposition}

\begin{proof}
For each \(1\le i\le |V^\circ|\),
\[
\frac{\partial L_i}{\partial s_i}
=
\sum_{P\in F_{\{v_i\}}}
\frac{\partial L_{i,P}}{\partial s_i}.
\]
By Proposition~\ref{prop:potential_convex_local},
\[
\frac{\partial L_{i,P}}{\partial s_i}>0
\qquad \text{for every } P\in F_{\{v_i\}}.
\]
Hence \(\frac{\partial L_i}{\partial s_i}>0\).

Now let \(1\le i\neq j\le |V^\circ|\). Then
\[
\frac{\partial L_i}{\partial s_j}
=
\sum_{P\in F_{\{v_i\}}}
\frac{\partial L_{i,P}}{\partial s_j}.
\]
If \(v_j\in V(P)\), then Proposition~\ref{prop:potential_convex_local} gives
\[
\frac{\partial L_{i,P}}{\partial s_j}<0;
\]
if \(v_j\notin V(P)\), then
\[
\frac{\partial L_{i,P}}{\partial s_j}=0.
\]
Therefore,
\[
\frac{\partial L_i}{\partial s_j}\le 0
\qquad \text{for } i\neq j.
\]

For each \(1\le i\le |V^\circ|\), we have
\[
\frac{\partial L_i}{\partial s_i}
-
\sum_{j\neq i}
\left|
\frac{\partial L_j}{\partial s_i}
\right|
=
\sum_{v\in V^\circ}\frac{\partial L_v}{\partial s_i}.
\]
Using \(L_v=\sum_{P\in F_{\{v\}}}L_{v,P}\) and exchanging the order of summation, we obtain
\[
\sum_{v\in V^\circ}\frac{\partial L_v}{\partial s_i}
=
\sum_{P\in F_{\{v_i\}}}
\sum_{v\in V(P)\cap V^\circ}
\frac{\partial L_{v,P}}{\partial s_i}.
\]
By Proposition~\ref{prop:potential_convex_local},
\[
\sum_{v\in V(P)\cap V^\circ}\frac{\partial L_{v,P}}{\partial s_i}>0
\qquad \text{for every } P\in F_{\{v_i\}}.
\]
Hence
\[
\frac{\partial L_i}{\partial s_i}
-
\sum_{j\neq i}
\left|
\frac{\partial L_j}{\partial s_i}
\right|
>0.
\]

Therefore \(M\) is symmetric, strictly diagonally dominant, and has positive diagonal entries. It follows that \(M\) is positive definite.
\end{proof}

\begin{corollary}
\label{cor2.3}
The potential function \(\mathcal E\) is strictly convex on \(\mathbb R^{|V^\circ|}\).
\end{corollary}

\begin{proof}
By Proposition~\ref{prop:potential_convex_g}, the matrix \(\mathrm{Hess}\,\mathcal E\) is positive definite. Hence \(\mathcal E\) is strictly convex on \(\mathbb R^{|V^\circ|}\).
\end{proof}


\section{Existence and Rigidity of Generalized Circle Packings}
\label{sec:existence_rigidity}

\begin{lemma}
\label{lem:limit_total_curvature}
(\cite{hu2025combinatorial}, Lemmas 3.2 and 3.3)
Let \(P\in F\) be a face, and let \(\{s^m=(s_1^m,\dots,s_{N(P)}^m)\}_{m=1}^{+\infty}\subset \mathbb R^{N(P)}\) be a sequence. For each \(m\), let \(\widetilde P^{\,m}\) be the generalized conical polygon associated with the face \(P\) and the geodesic curvatures \(k_i^m=e^{s_i^m}\), \(i=1,\dots,N(P)\). Let \(L_{i,P^m}\) be the total geodesic curvature of the arc \(C_i^m\cap P^m\).

Assume that \(s^m\to c=(c_1,\dots,c_{N(P)})\in[-\infty,+\infty]^{N(P)}\) as \(m\to+\infty\). Then:

\begin{itemize}
    \item[(I)] If \(c_j=-\infty\) for some \(j\in\{1,\dots,N(P)\}\), then
    \[
    \lim_{m\to+\infty} L_{j,P^m}=0.
    \]

    \item[(II)] Let
    \[
    I=\{\,i\in\{1,\dots,N(P)\}\mid c_i=+\infty\,\}.
    \]
    If \(I\neq\emptyset\), then
    \[
    \lim_{m\to+\infty}\sum_{i\in I}L_{i,P^m}
    =
    \begin{cases}
    |I|\,\pi, & \text{if } |I|<N(P)-2+\dfrac{Y_P}{\pi},\\[6pt]
    (N(P)-2)\pi+Y_P, & \text{if } |I|\ge N(P)-2+\dfrac{Y_P}{\pi}.
    \end{cases}
    \]
\end{itemize}
\end{lemma}

\begin{lemma}
\label{lem:smooth_embedding}
(\cite{hu2025hyperbolic}, Lemma 3.6)
Let \(\mathcal F:\mathbb R^n\to\mathbb R\) be a \(C^2\)-smooth strictly convex function whose Hessian is positive definite. Then its gradient map \(\nabla\mathcal F:\mathbb R^n\to\mathbb R^n\) is a smooth embedding.
\end{lemma}

We now prove Theorem~\ref{thm:existence_uniqueness_gcp}.

\begin{proof}
The argument follows the same idea as in \cite{hu2025combinatorial}, but here we work only with the interior variables on \(V^\circ\).

Recall that
\[
\mathcal L=
\left\{
(L_1,\dots,L_{|V^\circ|})\in\mathbb R_{>0}^{|V^\circ|}
\;\middle|\;
\sum_{v\in W}L_v
<
\sum_{P\in F_W}
\pi\min\left\{N(P,W),\,N(P)-2+\frac{Y_P}{\pi}\right\},
\ \forall\, W\subset V^\circ
\right\}.
\]

By the Gauss--Bonnet theorem and Lemma~\ref{lem:limit_total_curvature}, we have
\[
\nabla\mathcal E(\mathbb R^{|V^\circ|})\subset \mathcal L.
\]
By Corollary~\ref{cor2.3} and Lemma~\ref{lem:smooth_embedding}, the map \(\nabla\mathcal E\) is a smooth embedding. Hence \(\nabla\mathcal E(\mathbb R^{|V^\circ|})\) is an open subset of \(\mathcal L\). By Brouwer's theorem on invariance of domain, it remains to analyze the boundary behavior of \(\nabla\mathcal E\).

Take a sequence \(\{s^m\}_{m=1}^{\infty}\subset\mathbb R^{|V^\circ|}\) such that
\[
s^m\to a=(a_1,\dots,a_{|V^\circ|})\in[-\infty,+\infty]^{|V^\circ|},
\qquad m\to+\infty,
\]
where at least one component of \(a\) is equal to \(+\infty\) or \(-\infty\). We show that \(\nabla\mathcal E(s^m)\) converges to \(\partial\mathcal L\).

Define
\[
W_+=\{\,v_i\in V^\circ\mid a_i=+\infty\,\},
\qquad
W_-=\{\,v_i\in V^\circ\mid a_i=-\infty\,\}.
\]

Assume first that \(W_+\neq\emptyset\). For each face \(P\in F_{W_+}\), Lemma~\ref{lem:limit_total_curvature} gives
\[
\lim_{m\to+\infty}\sum_{v_i\in V(P)\cap W_+}L_{i,P}(s^m)
=
\pi\min\left\{N(P,W_+),\,N(P)-2+\frac{Y_P}{\pi}\right\}.
\]
Summing over all faces in \(F_{W_+}\), we obtain
\[
\lim_{m\to+\infty}\sum_{v_i\in W_+}L_i(s^m)
=
\sum_{P\in F_{W_+}}
\pi\min\left\{N(P,W_+),\,N(P)-2+\frac{Y_P}{\pi}\right\}.
\]
Thus \(\nabla\mathcal E(s^m)\) converges to the boundary of \(\mathcal L\).

Now assume that \(W_-\neq\emptyset\). For any \(v_i\in W_-\), Lemma~\ref{lem:limit_total_curvature} gives
\[
\lim_{m\to+\infty}L_i(s^m)
=
\lim_{m\to+\infty}\sum_{P\in F_{\{v_i\}}}L_{i,P}(s^m)
=
0.
\]
Again, \(\nabla\mathcal E(s^m)\) converges to the boundary of \(\mathcal L\).

Therefore every sequence approaching the boundary of \(\mathbb R^{|V^\circ|}\) is mapped to the boundary of \(\mathcal L\). Since \(\nabla\mathcal E(\mathbb R^{|V^\circ|})\) is a nonempty open subset of \(\mathcal L\), it follows that
\[
\nabla\mathcal E(\mathbb R^{|V^\circ|})=\mathcal L.
\]
Hence for every prescribed \((L_1,\dots,L_{|V^\circ|})\in\mathcal L\), there exists a unique \(s\in\mathbb R^{|V^\circ|}\) such that
\[
\nabla\mathcal E(s)=(L_1,\dots,L_{|V^\circ|}).
\]
Equivalently, there exists a unique generalized circle packing metric realizing the prescribed interior total geodesic curvatures, the prescribed boundary geodesic curvatures, and the prescribed discrete Gaussian curvatures.
\end{proof}



\section{Discrete Schwarz--Pick lemma}
\label{sec:schwarz_pick_lemma}

In this section, we use the variational results obtained above to prove the discrete Schwarz--Pick lemma. The argument is similar to that in \cite{HLLY}. We first recall the notation.

For \(j=1,2\), let \(\hat k_i^j\) be the prescribed geodesic curvature on the boundary, and let \(\hat L_i\) be the prescribed total geodesic curvature at interior vertices, where \((\hat L_1,\dots,\hat L_{|V^\circ|})^T\in \mathcal L\). Let \(\mathcal P^j\) be the generalized circle packing on \((S_{g,n},\Sigma_{g,n})\) realizing these data. Denote by \(k_i^j\) the geodesic curvature of \(\mathcal P^j\) at \(v_i\in V^\circ\), and set \(k_i^j=\hat k_i^j\) for every \(v_i\in V^\partial\).

We first prove a discrete maximum principle for the curvature ratio.

\begin{lemma}[Discrete maximum principle for the curvature ratio]
\label{lem:maximum}
Let \(\mathcal P^1\) and \(\mathcal P^2\) be two generalized circle packings on \((S_{g,n},\Sigma_{g,n})\) realizing the same prescribed interior total geodesic curvatures \((\hat L_1,\dots,\hat L_{|V^\circ|})^T\in\mathcal L\). For \(j=1,2\), let \(k_i^j\) be the geodesic curvature of \(\mathcal P^j\) at \(v_i\in V\). If
\[
\max_{v_i\in V}\frac{k_i^1}{k_i^2}>1,
\]
then this maximum cannot be attained at any interior vertex \(v_i\in V^\circ\).
\end{lemma}

\begin{proof}
Assume, to the contrary, that there exists an interior vertex \(v_0\in V^\circ\) such that
\[
\frac{k_0^1}{k_0^2}
=
\max_{v_i\in V}\frac{k_i^1}{k_i^2}
>1.
\]
Set \(u_i=\ln\frac{k_i^1}{k_i^2}\) for \(v_i\in V\). Then \(u_0>0\) and \(u_0\ge u_i\) for all \(v_i\in V\).

For \(t\in[0,1]\), define \(s_i(t)=(1-t)\ln k_i^2+t\ln k_i^1\), and let \(L_0(t)\) be the total geodesic curvature at \(v_0\) corresponding to the data \(s(t)\). Since \(\mathcal P^1\) and \(\mathcal P^2\) realize the same prescribed interior total geodesic curvatures, we have \(L_0(0)=L_0(1)=\hat L_0\). Hence, by the mean value theorem, there exists \(\xi\in(0,1)\) such that
\begin{equation}
\label{mean_eq}
0=L_0(1)-L_0(0)
=
\sum_{v_i\in V}\frac{\partial L_0}{\partial s_i}\Big|_{t=\xi}\,u_i.
\end{equation}

For each face \(P\in F_{\{v_0\}}\), let \(L_{i,P}\) be the contribution of \(P\) to the total geodesic curvature at \(v_i\in V(P)\). Then \(L_0=\sum_{P\in F_{\{v_0\}}}L_{0,P}\). By locality, for each \(v_i\neq v_0\),
\[
\frac{\partial L_0}{\partial s_i}
=
\sum_{\substack{P\in F_{\{v_0\}}\cap F_{\{v_i\}}}}
\frac{\partial L_{0,P}}{\partial s_i}.
\]
Using the symmetry of the local Hessian, we obtain
\[
\frac{\partial L_0}{\partial s_i}
=
\sum_{\substack{P\in F_{\{v_0\}}\cap F_{\{v_i\}}}}
\frac{\partial L_{i,P}}{\partial s_0},
\qquad v_i\neq v_0.
\]
Substituting this into \eqref{mean_eq} and regrouping the terms face by face, we get
\[
0=
\sum_{P\in F_{\{v_0\}}}
\left[
\frac{\partial L_{0,P}}{\partial s_0}\,u_0
+
\sum_{v_i\in V(P)\setminus\{v_0\}}
\frac{\partial L_{i,P}}{\partial s_0}\,u_i
\right].
\]

Since \(u_i\le u_0\) for all \(v_i\in V\), and \(\frac{\partial L_{i,P}}{\partial s_0}<0\) for every \(v_i\in V(P)\setminus\{v_0\}\), it follows that
\[
\frac{\partial L_{i,P}}{\partial s_0}\,u_i
\ge
\frac{\partial L_{i,P}}{\partial s_0}\,u_0.
\]
Therefore, for each \(P\in F_{\{v_0\}}\),
\[
\frac{\partial L_{0,P}}{\partial s_0}\,u_0
+
\sum_{v_i\in V(P)\setminus\{v_0\}}
\frac{\partial L_{i,P}}{\partial s_0}\,u_i
\ge
\left(
\sum_{v_i\in V(P)}
\frac{\partial L_{i,P}}{\partial s_0}
\right)u_0.
\]
By the positivity of the local column sum and the fact that \(u_0>0\), each bracketed term is strictly positive. Hence the right-hand side of \eqref{mean_eq} is strictly positive, a contradiction.

Therefore, if \(\max_{v_i\in V}\frac{k_i^1}{k_i^2}>1\), then this maximum cannot be attained at any interior vertex \(v_i\in V^\circ\).
\end{proof}

We now prove the four Schwarz--Pick theorems. Assume that \(\hat k_i^1\le \hat k_i^2\) for all \(v_i\in V^\partial\).

\begin{theorem}[Schwarz--Pick I]
\label{thm:schwarz_pick_I}
Let the generalized circle packings \(\mathcal P^1\) and \(\mathcal P^2\) be defined as above. Then \(k_i^1\le k_i^2\) for every vertex \(v_i\in V\). Moreover, if equality occurs at some interior vertex \(v_i\in V^\circ\), then \(k_i^1=k_i^2\) for all \(v_i\in V\).
\end{theorem}

\begin{proof}
Set \(u_i=\ln\frac{k_i^1}{k_i^2}\) for \(v_i\in V\). Then the boundary assumption is equivalent to \(u_i\le 0\) for all \(v_i\in V^\partial\).

Suppose that there exists a vertex \(v_i\in V\) such that \(k_i^1>k_i^2\). Then \(\max_{v_i\in V}\frac{k_i^1}{k_i^2}>1\). By Lemma~\ref{lem:maximum}, this maximum cannot be attained at any interior vertex. Hence it must be attained at some boundary vertex \(v_i\in V^\partial\). However, on the boundary we have
\[
\frac{k_i^1}{k_i^2}=\frac{\hat k_i^1}{\hat k_i^2}\le 1,
\]
which is impossible. Therefore \(\frac{k_i^1}{k_i^2}\le 1\) for all \(v_i\in V\), that is, \(k_i^1\le k_i^2\) for all \(v_i\in V\).

We now prove the rigidity statement. Assume that there exists an interior vertex \(v_0\in V^\circ\) such that \(k_0^1=k_0^2\). Since we have already proved that \(k_i^1\le k_i^2\) for all \(v_i\in V\), it follows that \(u_i\le 0\) for all \(v_i\in V\), and \(u_0=0\). Thus \(u_0\) is the maximum of \(u\) on \(V\).

Applying the same computation as in the proof of Lemma~\ref{lem:maximum} at the vertex \(v_0\), we obtain
\[
0=
\sum_{P\in F_{\{v_0\}}}
\sum_{v_i\in V(P)\setminus\{v_0\}}
\frac{\partial L_{i,P}}{\partial s_0}\,u_i.
\]
For every \(P\in F_{\{v_0\}}\) and every \(v_i\in V(P)\setminus\{v_0\}\), we have \(\frac{\partial L_{i,P}}{\partial s_0}<0\) and \(u_i\le 0\). Hence each term \(\frac{\partial L_{i,P}}{\partial s_0}\,u_i\) is nonnegative. Since their total sum is zero, every term must vanish. Therefore \(u_i=0\) for every vertex \(v_i\) belonging to a face in \(F_{\{v_0\}}\).

Repeating this argument and using the connectedness of the \(1\)-skeleton, we conclude that \(u_i=0\) for all \(v_i\in V\). Hence \(k_i^1=k_i^2\) for all \(v_i\in V\).
\end{proof}

\begin{theorem}[Schwarz--Pick II]
\label{thm:schwarz_pick_II}
Let the generalized circle packings \(\mathcal P^1\) and \(\mathcal P^2\) be defined as above. Then \(l_{i,P}^1\ge l_{i,P}^2\) for every face \(P\in F\) and every vertex \(v_i\in V(P)\), where \(l_{i,P}^j\) denotes the length of the generalized circle arc centered at \(v_i\) in the face \(P\) for \(j=1,2\).

Moreover, if equality occurs at some generalized circle arc centered at an interior vertex \(v_i\in V^\circ\), that is, if \(l_{i,P}^1=l_{i,P}^2\) for some face \(P\in F\) with \(v_i\in V(P)\), then \(k_i^1=k_i^2\) for all \(v_i\in V\).
\end{theorem}

\begin{proof}
By Theorem~\ref{thm:schwarz_pick_I}, we have \(k_i^1\le k_i^2\) for all \(v_i\in V\). Fix a face \(P\in F\). By the argument in the proof of Lemma 2.6 in \cite{hu2025hyperbolic}, the curvature \(k_P\) of the dual circle is increasing with respect to the vertex curvatures of \(P\). Hence \(k_P^1\le k_P^2\), where \(k_P^j\) is the geodesic curvature of the dual circle of \(P\) in \(\mathcal P^j\).

Now fix a vertex \(v_i\in V(P)\). The local arc length is a smooth function of \(k_i\) and \(k_P\), say \(l_{i,P}=l_{i,P}(k_i,k_P)\). By the proof of Theorem 1.6 in \cite{hu2025combinatorial},
\[
\frac{\partial l_{i,P}}{\partial k_i}<0,
\qquad
\frac{\partial l_{i,P}}{\partial k_P}<0.
\]
Thus \(l_{i,P}(k_i,k_P)\) is strictly decreasing in both variables, and so
\[
l_{i,P}^1=l_{i,P}(k_i^1,k_P^1)\ge l_{i,P}(k_i^2,k_P^2)=l_{i,P}^2.
\]

Assume now that \(l_{i,P}^1=l_{i,P}^2\) for some face \(P\in F\) and some interior vertex \(v_i\in V^\circ\cap V(P)\). Since \(l_{i,P}(k_i,k_P)\) is strictly decreasing in both variables, the equality together with \(k_i^1\le k_i^2\) and \(k_P^1\le k_P^2\) implies \(k_i^1=k_i^2\). As \(v_i\in V^\circ\), the rigidity part of Theorem~\ref{thm:schwarz_pick_I} yields \(k_i^1=k_i^2\) for all \(v_i\in V\).
\end{proof}

\begin{theorem}[Schwarz--Pick III]
\label{thm:schwarz_pick_III}
Let the generalized circle packings \(\mathcal P^1\) and \(\mathcal P^2\) be defined as above. Then
\[
\rho^1(v_i,v_l)\ge \rho^2(v_i,v_l)
\]
for any two vertices \(v_i,v_l\in \widetilde V\), where \(\widetilde V=\{v_i\in V\mid k_i^1>1\}\). Here \(\rho^1\) and \(\rho^2\) denote the distances defined by admissible arc chains associated with the polygonal cellular decomposition on the hyperbolic surfaces \(\widetilde S_{g,n}^1\) and \(\widetilde S_{g,n}^2\), respectively.

Moreover, if equality occurs for some \(v_i,v_l\in \widetilde V\), with at least one of them in \(V^\circ\), then \(k_i^1=k_i^2\) for all \(v_i\in V\).
\end{theorem}

\begin{proof}
Let \(v_i,v_l\in \widetilde V\). By Theorem~\ref{thm:schwarz_pick_I}, we have
\[
k_i^1\le k_i^2,\qquad \forall\, v_i\in V.
\]
In particular, since \(k_i^1>1\) for every \(v_i\in \widetilde V\), it follows that
\[
k_i^2\ge k_i^1>1.
\]
Hence the vertices in \(\widetilde V\) correspond to genuine hyperbolic circles in both
\(\widetilde S_{g,n}^1\) and \(\widetilde S_{g,n}^2\).

Let \(\mathcal C_{i,l}\) be the set of all admissible arc chains connecting the boundary
components corresponding to \(v_i\) and \(v_l\). For each \(\gamma\in\mathcal C_{i,l}\), let
\[
L^j(\gamma)=\sum_{\alpha\in\gamma} l_\alpha^j,\qquad j=1,2,
\]
be the total length of \(\gamma\) in \(\mathcal P^j\). By Theorem~\ref{thm:schwarz_pick_II}, every generalized circle arc satisfies
\[
l_\alpha^1\ge l_\alpha^2.
\]
Therefore,
\[
L^1(\gamma)\ge L^2(\gamma),\qquad \forall\, \gamma\in\mathcal C_{i,l}.
\]

By Definition~\ref{def:generalized_circle_packing}, the distance \(\rho^j(v_i,v_l)\) is the minimum of \(L^j(\gamma)\) over all \(\gamma\in\mathcal C_{i,l}\). Hence
\[
\rho^1(v_i,v_l)
=
\min_{\gamma\in\mathcal C_{i,l}}L^1(\gamma)
\ge
\min_{\gamma\in\mathcal C_{i,l}}L^2(\gamma)
=
\rho^2(v_i,v_l).
\]
This proves the distance comparison.

We now prove the rigidity statement. Assume that for some \(v_i,v_l\in \widetilde V\), with at least one of them in \(V^\circ\),
\[
\rho^1(v_i,v_l)=\rho^2(v_i,v_l).
\]
Choose an admissible arc chain \(\gamma^\ast\in\mathcal C_{i,l}\) such that
\[
L^1(\gamma^\ast)=\rho^1(v_i,v_l).
\]
Then
\[
\rho^1(v_i,v_l)
=
L^1(\gamma^\ast)
\ge
L^2(\gamma^\ast)
\ge
\rho^2(v_i,v_l).
\]
Since \(\rho^1(v_i,v_l)=\rho^2(v_i,v_l)\), all the above inequalities are equalities. In particular,
\[
L^1(\gamma^\ast)=L^2(\gamma^\ast).
\]
Because each arc in \(\gamma^\ast\) satisfies \(l_\alpha^1\ge l_\alpha^2\), the equality of the total lengths implies
\[
l_\alpha^1=l_\alpha^2
\qquad \text{for every } \alpha\in\gamma^\ast.
\]

Since at least one of the endpoints belongs to \(V^\circ\), the arc chain \(\gamma^\ast\) contains a generalized circle arc centered at an interior vertex. Therefore, by the rigidity part of Theorem~\ref{thm:schwarz_pick_II}, we conclude that
\[
k_i^1=k_i^2,\qquad \forall\, v_i\in V.
\]
This completes the proof.
\end{proof}

\begin{theorem}[Schwarz--Pick IV]
\label{thm:schwarz_pick_IV}
Let the generalized circle packings \(\mathcal P^1\) and \(\mathcal P^2\) be defined as above. Then
\[
\mathrm{Area}^1(\Omega_P)\ge \mathrm{Area}^2(\Omega_P)
\]
for every face \(P\in F\), where \(\mathrm{Area}^j(\Omega_P)\) denotes the hyperbolic area of the ideal region corresponding to the face \(P\) in \(\mathcal P^j\), \(j=1,2\).

Moreover, if equality occurs at some face \(P\in F\) containing an interior vertex, then \(k_i^1=k_i^2\) for all \(v_i\in V\).
\end{theorem}

\begin{proof}
By Theorem~\ref{thm:schwarz_pick_I}, we have \(k_i^1\le k_i^2\) for all \(v_i\in V\). Fix a face \(P\in F\). By Lemma 2.6 in \cite{hu2025hyperbolic}, the area \(\mathrm{Area}(\Omega_P)\) is a smooth function of the vertex curvatures \(\{k_i\}_{v_i\in V(P)}\), and
\[
\frac{\partial \mathrm{Area}(\Omega_P)}{\partial k_i}<0
\qquad \text{for every } v_i\in V(P).
\]
Hence \(\mathrm{Area}(\Omega_P)\) is strictly decreasing in each vertex curvature, and therefore \(\mathrm{Area}^1(\Omega_P)\ge \mathrm{Area}^2(\Omega_P)\) for every \(P\in F\).

Assume now that \(\mathrm{Area}^1(\Omega_P)=\mathrm{Area}^2(\Omega_P)\) for some face \(P\in F\) containing an interior vertex. Since \(\mathrm{Area}(\Omega_P)\) is strictly decreasing in each variable and \(k_i^1\le k_i^2\) for all \(v_i\in V(P)\), the equality of areas implies \(k_i^1=k_i^2\) for all \(v_i\in V(P)\). In particular, there exists an interior vertex \(v_0\in V^\circ\cap V(P)\) such that \(k_0^1=k_0^2\). The rigidity part of Theorem~\ref{thm:schwarz_pick_I} now gives \(k_i^1=k_i^2\) for all \(v_i\in V\).
\end{proof}

\section{Acknowledgments}
Yu Sun is supported by NSF of China, No.12501097 and University level natural science foundation of Nanjing Institute
of Technology, No.3534113223051. Zunwu He is supported by NSFC, No.12301094 and Guangzhou Basic and Applied Basic Research Foundation, No.2024A04J3483, and the General Program of Guangdong Basic and Applied Basic Research Foundation, No.2025A1515010502.

	\bibliographystyle{plain}
	\bibliography{cite1} 
	\Addresses

\end{document}